\documentclass[12pt]{amsart}
\usepackage{amsthm,amsmath,amssymb,latexsym,graphics}
\usepackage{hyperref}
\newtheorem{proposition}{Proposition}
\newtheorem{theorem}{Theorem}
\newtheorem{remark}{Remark}

\newtheorem*{claim}{Claim}
\newtheorem{cor}{Corollary}
\newcommand{\grad}{\nabla}

\newcommand{\vo}{d\/vol_{g_c}}
\newcommand{\Ks}{K^{[s]}}
\newcommand{\Fs}{F^{[s]}}
\newcommand{\As}{A^{[s]}}
\newcommand{\Ls}{\Lambda^{[s]}}
\newcommand{\ind}{\text{ind}}
\newcommand{\dxi}{\dot{\xi}}
\newcommand{\dv}{\dot{v}}
\newcommand{\mint}{\diagup\!\!\!\!\!\!\!\int_{\mathbb S^{4}}}
\newcommand{\ep}{\epsilon}

\begin{document}
\title{On the prescribing $\sigma_2$ curvature equation on $\mathbb S^4$}
\author{S.-Y. Alice  Chang}
\address{Department of Mathematics\\  Princeton University\\  Princeton, NJ 08540}
\email{chang@math.princeton.edu}
\author{Zheng-Chao Han }
\address{Department of Mathematics\\ Rutgers University\\ 110 Frelinghuysen Road\\
Piscataway, NJ 08854}
\email{zchan@math.rutgers.edu} 
\author{Paul Yang}
\address{Department of Mathematics\\  Princeton University\\  Princeton, NJ 08540}
\email{yang@math.princeton.edu}
\thanks{The research of the first and third author is 
partially supported  by  NSF through grant DMS-0758601;
the first author also gratefully acknowledges partial support
from the Minerva Research Foundation and The Charles Simonyi Endowment fund 
during the academic year 08-09 while visiting Institute of Advanced Study, Princeton.
The research of the second author is partially supported  by  NSF through grant DMS-0103888 and by
a Rutgers University Research Council Grant(202132).}
\date{}
\begin{abstract}
Prescribing $\sigma_k$ curvature equations are fully nonlinear generalizations of the 
prescribing Gaussian or scalar curvature equations. Given a positive function $K$ to be prescribed on the
$4$-dimensional round sphere. We obtain asymptotic profile analysis  for potentially blowing up 
solutions to the $\sigma_2$ curvature equation with the given $K$; and rule out the possibility of
 blowing up solutions  when $K$ satisfies a non-degeneracy condition. We also prove  uniform
a priori estimates for solutions to a family of  $\sigma_2$ curvature equations deforming $K$ to 
a positive constant under the same non-degeneracy condition on $K$, and prove the existence of a solution
using degree argument to this deformation involving fully nonlinear elliptic operators under an additional,
natural degree condition on a finite dimensional map associated with $K$.
\end{abstract}

\maketitle

\section{ Description of main results}
\vskip10pt

Our main results in this paper are  (potential) blow up profile analysis,  a priori estimates and existence of
 admissible solutions $w$ to 
the $\sigma_2$ curvature equation
\begin{equation} \label{1}
\sigma_2(g^{-1} \circ A_g) = K(x),
\end{equation}
on $\mathbb S^4$, 
where $g = e^{2w(x)}g_c$ is a metric conformal to $g_c$, 
with $g_c$ being the canonical background metric on the round sphere  $\mathbb S^4$,
$ A_g $ is the the Weyl-Schouten tensor of the  metric $g$, 
\begin{equation} \label{trans}
\begin{split}
A_g &=  \frac {1}{n-2} \{ Ric - \frac {R}{2(n-1)}g \} \\
& = A_{g_c} - \left[ \grad^2w-dw \otimes dw + \frac 12 |\grad w|^2 g_c\right],
\end{split}
\end{equation}
and $\sigma_k(\Lambda)$, for any $1-1$ tensor $\Lambda$ on an $n-$dimensional  vector space
and $k \in \mathbb N$, $0 \leq k \leq n$, is 
the $k$-th elementary symmetric function of the eigenvalues of $\Lambda$;
$K(x)$ is a given function on $\mathbb S^4$ with  some appropriate  assumptions, and
an \emph{admissible} solution is defined to be a $C^2(M)$ solution $w$ to
\eqref{1} such that for all $x \in \mathbb S^4$,
$A_g (x) \in \Gamma^+_k$, namely, $\sigma_j(g^{-1} \circ A_g ) >0$ for $1 \leq j \leq k$.
 Note that
 $\sigma_1(A_g)$ is simply a positive constant multiple of the scalar curvature of $g$, so
 $A_g$ in the $\Gamma^{+}_k$ class is a generalization of the notion that  the scalar
curvature $R_g$ of $g$ having a fixed $+$ sign. 

Note that, since
\[
\sigma_2(g^{-1} \circ A_g) = e^{-4w} \sigma_2(g_c^{-1} \circ A_g),
\]
 so  \eqref{1} is equivalent to
\begin{equation} \label{2}
 \sigma_2(g_c^{-1} \circ \left[A_{g_c} -  \grad^2w+dw \otimes dw - 
\frac 12 |\grad w|^2 g_c   \right]) \\
=  K(x) e^{4 w(x)}.
\end{equation}

It is well known that \eqref{2} is  \emph{elliptic} at an admissible solution;
and in fact, any solution $w$ to \eqref{2} on $\mathbb S^4$ is admissible.
There have been a large number of papers on problems related to the $\sigma_k$ curvature since
the  work \cite{V99} of Viaclovsky a decade ago. It is inadequate to do even
 a short survey of recent work in the introductory remarks here. We will instead refer the reader to
 recent surveys \cite{V06} by  Viaclovsky and \cite{CC07} by Chang and Chen.

As alluded to above, a similar problem to \eqref{2} for the $\sigma_1$ curvature was a predecessor to
\eqref{2}. More specifically, 
 if we  prescribe a function $K(x)$ on the round $n$-dimensional sphere $(\mathbb S^n,  g_c)$ to be the
scalar curvature of a metric $g=e^{2w}g_c$ pointwise conformal to $g_c$, then $w$ satisfies
\begin{equation}\label{sc}
2(n-1)\Delta_{g_c} w +(n-1)(n-2)|\grad w|^2=R_{g_c}-K(x) e^{2w}.
\end{equation}
Similar equation can be formulated for a general manifold. 
One difference between \eqref{sc} and \eqref{2} is that
\eqref{sc} is semilinear in $w$, while \eqref{2} is fully nonlinear in $w$.
\eqref{sc} takes on the familiar form 
\begin{equation}\label{gc}
2\Delta_{g_c} w =R_{g_c}-K(x)e^{2w},
\end{equation}
when $n=2$,  which is the Nirenberg problem. The $n\ge 3$ case of \eqref{sc} is often written in terms
of a different variable $u=e^{(n-2)w/2}$, which would render the equation in the familiar form
\begin{equation}\label{ya}
-4\frac{n-1}{n-2} \Delta_{g_c} u + R_{g_c} u = K(x)u^{\frac{n+2}{n-2}}.
\end{equation}
The $K(x)\equiv \text{const.}$ case of \eqref{ya} on a general compact manifold
is the famous Yamabe problem. \eqref{gc} and \eqref{ya}
have attracted enormous attention in the last several decades. A large collection of phenomena on the
possible behavior of solutions to these equations, and methods and techniques of attacking these problems have
been accumulated, which have tremendously enriched our understanding in solving a large class of nonlinear
(elliptic) PDEs, and provided guidance in attacking seemingly unrelated problems. 
It is impossible in the space here to provide even a partial list of references. Please see 
\cite{K85}, \cite{LP}, \cite{S91}, \cite{L98}, \cite{CY02}, \cite{KMS09}, and the references therein
 to get a glimpse of the results and techniques in this area.

Directly related to our current work are some work on the (potential) blow up analysis, a priori estimates, 
and existence of solutions to \eqref{gc} or \eqref{ya}. It is proved in \cite{Han89} and \cite{CGY93} that
when $K$ is a positive function on $\mathbb S^2$, a sequence of blowing up solutions to \eqref{gc} 
has only one point blow up and has a well-defined blow up profile, and that when
$K$ is a positive $C^2$ function on $\mathbb S^2$ such that $\Delta_{g_c} K(x)\ne 0$ at any of its
critical points, no blow up can happen, more precisely, there is an a priori bound on the set of solutions to
\eqref{gc} which depends on the $C^2$ norm of $K$, the positive lower bound of $K$ and $|\Delta_{g_c} K(x)|$ 
near the critical points of $K$, and the modulus of continuity of the second derivatives of $K$. Similar
results for \eqref{ya} in the case $n=3$ were proved in \cite{Z90} and \cite{CGY93}, and for
\eqref{ya} in the case $n\ge 4$ in \cite{YL95} under a flatness condition of $K$ at its critical points.
When $K$ is a Morse function, say, this flatness condition fails when $n\ge 4$. In fact it is proved in
\cite{YL96} that in such cases on $\mathbb S^4$, there can be a sequence of  solutions to \eqref{ya}
blowing up at more than one points.  Later on \cite{CL99} constructed solutions blowing up on $\mathbb S^n$,
$n\ge 7$, 
with unbounded layers of ``energy concentration" for certain non-degenerate $K$. A natural question 
concerning the $\sigma_k$ curvature equations such as \eqref{2} is:
 which kind of behavior does its solution exhibit?

In the following we will often transform \eqref{2} through a conformal automorphism $\varphi$ of $\mathbb S^4$
as follows. Let $|d \varphi (P)|$ denote the factor such that $|d \varphi (P)[X]|= |d \varphi (P)| |X|$ for
any tangent vector $X\in T_P(\mathbb S^4)$, and 
\begin{equation}\label{wtransf}
w_{\varphi}(P)= w\circ \varphi (P) + \ln |d \varphi (P)|.
\end{equation}
Then $w$ is a solution to \eqref{2} iff  $w_{\varphi}$ is a solution to
\begin{equation}\label{transf}
 \sigma_2(g_c^{-1} \circ \left[A_{g_c} -  \grad^2w_{\varphi} +dw_{\varphi}  \otimes dw_{\varphi}  - 
\frac 12 |\grad w_{\varphi} |^2 g_c   \right]) \\
=  K\circ \varphi (x) e^{4 w_{\varphi} (x)}.
\end{equation}

Our first results show that  solutions to
\eqref{2} on $\mathbb S^4$ exhibits similar behavior as those to \eqref{gc} on $\mathbb S^2$ or
\eqref{ya} on $\mathbb S^3$.

\begin{theorem} \label{chy}
Consider a family of admissible conformal metrics $g_j=e^{2w_j}g_c$ on $\mathbb S^4$ with
$\sigma_2(g^{-1}_j\circ A_{g_j}) = K(x)$, where $g_c$ denotes
the canonical round metric on $\mathbb S^4$ and $K(x)$ denotes a  $C^2$ positive
 function on $\mathbb S^4$.
Then there exists at most one isolated simple blow up point in the sense that, 
if $\max w_j = w_j (P_j) \to \infty$, then there exists conformal automorphism $\varphi_j$ of $\mathbb S^4$
such that, if we define $v_j (P) = w_j \circ \varphi_j (P) + \ln |  d \varphi_j(P)|$, we have
\begin{equation}\label{ptws}
v_j (P) - \frac {1}{4} \ln \frac {6}{K(P_j)} \to 0 \quad \text{in} \quad L^{\infty}(\mathbb S^4),
\end{equation}
and
\begin{equation}\label{grad4}
\int_{\mathbb S^4} |\grad v_j|^4 \to 0.
\end{equation}
In fact, we have the stronger conclusion that the $W^{2,6}$ norm of $v_j$ stays
bounded and  $v_j -  \frac {1}{4} \ln \frac {6}{K(P_j)} \to 0$ in 
$C^{1, \alpha}(\mathbb S^4)$ for any $0< \alpha < 1/3$.
\end{theorem}
We also have

\begin{theorem}\label{chy2}
Let $K(x)$ be a $C^2$ positive function on $\mathbb S^4$  satisfying a non-degeneracy condition
\begin{equation}\label{nd}
\Delta K(P) \neq 0 \quad \text{whenever}\quad \grad K(P) =0,
\end{equation}
and we consider solutions $w(x)$ to \eqref{2}, with $K(x)$ replaced by 
$$\Ks (x) := (1-s)6+s K(x),$$
 for $0<s \le 1$, namely,
\begin{equation}\tag{\ref{2}$'$}\label{2'}
\sigma_2(g_c^{-1} \circ \left[A_{g_c} -  \grad^2w+dw \otimes dw - 
\frac 12 |\grad w|^2 g_c   \right]) \\
=  \Ks (x) e^{4 w(x)}.
\end{equation}
Then there exist a priori $C^{2, \alpha}$ estimates on $w$, uniform in $0<s \le 1$, which depend on  
the $C^2$ norm of
$K$, the modulus of continuity of $\grad^2 K$,  positive lower bound of  $\min K$
and positive lower bound of $|\Delta K(x)|$ in a neighborhood
of the critical points of $K$.
\end{theorem}
\begin{remark}
\emph{\textbf{Theorems 1}} and \emph{\textbf{2}}, with the uniform estimates in \emph{\textbf{Theorem 2}} for solutions
to \eqref{2'} only for $0<s_0\le s \le 1$ and the estimates possibly depending on
$0<s_0<1$, were obtained several years ago and were announced in
\cite{Han04}. The details were  written up in \cite{CHY1}
 and presented by the second author on several occasions, including at the
2006 Banff workshop ``Geometric and Nonlinear Analysis". 
The current work can be considered as a completion of \cite{CHY1}.
As mentioned above similar statements for \eqref{gc} and \eqref{ya} 
were obtained earlier in \cite{Han89}, \cite{CGY93}, \cite{YL95}, \cite{Z90}, among others. 
When applying these estimates to the corresponding equation such as \eqref{gc} and \eqref{ya}
with $K(x)$ replaced by $\Ks (x) $, 
all previous work stated and proved that the a priori  estimates on the solutions remain uniform as long as
$0<s_0\le s \le 1$, for any fixed $s_0$. This stems from the dependence of the
 a priori  estimates  on
  a positive lower bound of $|\Delta K(x)|$ near the critical points of $K$, among other things. 
Since $\Delta \Ks = s \Delta K(x)$ becomes small when $s>0$ is small, previous work in this area
assumed that the a priori  estimates could deteriorate as $s>0$ becomes small. In these previous work,
one has to devise a way to study the problem when  $s>0$ becomes small. \cite{CGY93} and \cite{YL95}
used some kind of ``center of mass" analysis 
via conformal transformations of the round sphere. Technically \cite{CGY93} and \cite{YL95} used a 
constrained variational problem to study the  ``centered problem". In essence the success of these
methods was due to the semilinear nature of the relevant equations, so one could still have control on the
``centered solution" in some norm weaker than $C^{2,\alpha}$ norm, say, $W^{2,p}$ norm, when $s>0$ is small,
and used these estimates to prove existence of solutions under natural geometric/topological
assumptions on $K$.
This approach was problematic for our \underline{fully nonlinear} equation \eqref{2}. Due to this difficulty,
until recently we have not been successful in   using our preliminary version of 
 \emph{\textbf{Theorem 2}} (for estimates in the range $0<s_0\le s \le s$ which may depend on $s_0$)
and the deformation $\Ks$ above to the equation to establish solutions to \eqref{2}, under natural 
geometric/topological 
assumptions on $K$. It was our recent realization that in our setting, as well as in those of 
\cite{CGY93} and  \cite{YL95},
the  a priori  estimates of solutions, under conditions like those in   \emph{\textbf{Theorem 2}},
remain uniform for all $1\ge s>0$\/! After we completed this work and were compiling the bibliography for this
paper, we noticed that M. Ji made  a similar observation in her work on \eqref{gc} in \cite{J04}.
This uniform a prioir estimates for   all $1\ge s>0$   leads to our next Theorem.
\end{remark}

\begin{theorem}\label{exist}
Suppose $K(x)$ is a  $C^2$ positive
 function on $\mathbb S^4$ satisfying \eqref{nd}. Then the map
\[
G(P,t) = |\mathbb S^4|^{-1} \int_{\mathbb S^4} K \circ \varphi_{P,t}(x) x \, \vo \in \mathbb R^5
\]
does not have a zero for $(P,t)\in \mathbb S^4\times [t_1, \infty)$, for $t_1$ large. 

Furthermore,  consider $G$ as a map defined on $(t-1)P/t\in B_r(O)$ for $r> (t-1)/t$ and if
\begin{equation} \label{de}
\deg \left(G, B_r(O), O\right) \neq 0, \quad \text{for $r\ge r_1=(t_1-1)/t_1$,}
\end{equation}
then \eqref{2} has a solution.

In particular, if $K$ has only isolated 
critical points in the region $\{x \in \mathbb S^4: \Delta K(x) <0\}$ and 
\[
\sum_{x \in \mathbb S^4: \Delta K(x) <0, \grad K(x) =0} \ind(\grad K(x)) \neq 1,
\]
where $\ind(\grad K(x))$ stands for the index of the vector field $\grad K(x)$ at its
isolated zero $x$, then \eqref{de} holds, therefore,
\eqref{2} has a solution.
\end{theorem}

A corollary of the proof for the $W^{2,p}$ estimate in 
Theorem~\ref{chy} is a bound on a functional determinant whose critical points are solutions to
\eqref{2}. We recall that the relevant functional determinant is defined, similar to \cite{CGY1}, through
\[
I\hskip -1mm I [w]  
=  \mint \left( (\Delta_0 w)^2 + 2 |\grad_0 w|^2 + 12 w \right)\,  d \/ vol_{g_c},
\]
\[
C_K[w]= 3  \log \left( \mint K e^{4w}\,  d \/ vol_{g_c}\right),
\]
and
\[
\begin{split}
Y[w] &= \frac {1}{36} \left(  \mint R^2 \,  d \/ vol_{g_c} -  \mint R^2_0 \,  d \/ vol_{g_c}\right) \\
&= \mint \left( \Delta_0 w + |\grad_0 w|^2\right)^2 \, d \/ vol_{g_c} - 4  \mint |\grad_0 w|^2 \, d  \/vol_{g_c}.
\end{split}
\]
$ F[w] = Y[w]- I\hskip -1mm I [w]+C_K[w]$ is the relevant functional determinant and a critical point of
$F[w]$ is a solution of \eqref{2}. 
It is known that, for any conformal transformation $\varphi$ of $(\mathbb S^4, g_c)$,  $ Y[w_{\varphi}] = Y[w]$, 
and $I\hskip -1mm I[w_{\varphi}]= I\hskip -1mm I[w]$. 

There is a similar functional determinant and a variational characterization for solutions to the 
prescribing Gaussian curvature problem on  $\mathbb S^2$. Chang, Gursky and Yang proved in \cite{CGY93} that
this functional is bounded on the set of solutions to the prescribing Gaussian curvature problem on  $\mathbb S^2$
for any  positive function $K$ on  $\mathbb S^2$ to be prescribed. Our corollary is in the same spirit.
\begin{cor}
Let $K(x)$ be a given positive  $C^2$ function on $\mathbb S^4$. Then there is a bound $C$ depending on $K$ only through
the $C^2$ norm of $K$, a positive upper and lower bound of $K$ on $\mathbb S^4$,
such that 
\[
| F[w]| \le C
\]
for all admissible solutions $w$ to \eqref{2}.
\end{cor}
 {\bf Theorem 1} will be established using blow up analysis, Liouville type classification results of entire solutions,
and integral type estimates for such fully nonlinear equations from  \cite{CGY1} and \cite{Han04}.
{\bf Theorem 2} will be established using a  weaker version of
{\bf Theorem 1} and  a Kazdan-Warner type identity satisfied by the solutions.
The weaker
version of {\bf Theorem 1}  only needs to establish
\begin{equation}\tag{\ref{ptws}$'$}\label{ptws'}
v_j(P)-\frac 14 \frac{6}{K(P_j)} \to 0 \quad \text{pointwise on $\mathbb S^4 \setminus \{-P_j\}$,
and bounded in $L^{\infty}(\mathbb S^4)$,}
\end{equation}
instead of \eqref{ptws}, \eqref{grad4} and the $W^{2,6}$ norm estimates.
A degree argument for a fully nonlinear operator
associated with \eqref{2} and {\bf Theorem 2} will be used to establish {\bf Theorem 3}. 
To streamline our presentation, we will first outline the main steps for proving {\bf Theorem 3}, 
assuming  {\bf Theorem 2} and all the other needed ingredients.
In the remaining sections, we will first provide a proof for  
\eqref{ptws'} in {\bf Theorem 1} and  for {\bf Theorem 2}, before
finally providing a proof for the $W^{2,6}$ norm estimates in {\bf Theorem 1} and for {\bf Corollary 1}.

\vskip12pt
\section{Proof of Theorem \ref{exist}}
\vskip10pt

The first and third parts of {\bf Theorem \ref{exist}} is contained in 
\cite{CY91} and \cite{CGY93}. 
We will establish the second part of {\bf Theorem \ref{exist}} by formulating
the existence of a solution to \eqref{2} as a degree problem for a nonlinear map and linking
the degree of this map to that of $G$.

By a fibration result from \cite{CY87}, \cite{CY91},
\cite{CL} and \cite{YL95}, see also \cite{A79}, \cite{O82} for
early genesis of these ideas, 
if we define
\[
\mathcal S_0 = \{ v \in C^{2, \alpha}(\mathbb S^4):\, \int_{\mathbb S^4} e^{4v(x)} x\, \vo = 0\},
\]
then the map
$\pi: (v, \xi) \in \mathcal S_0 \times B \mapsto C^{2, \alpha}(\mathbb S^4)$ defined by
\[
\pi (v, \xi) = v \circ \varphi_{P,t}^{-1} + \ln | d \varphi_{P,t}^{-1}|,
\]
with $B$ denoting the open unit ball in $\mathbb R^5$ and $\xi=rP$, $P \in \mathbb S^4$, $r=(t-1)/t$, $t\ge 1$,
is a $C^2$ diffeomorphism from $\mathcal S_0 \times B $ \emph{onto} $ C^{2, \alpha}(\mathbb S^4)$. 
Thus $(v, P, t)\in \mathcal S_0 \times \mathbb S^4 \times [1, \infty)$ provide global coordinates for 
$ C^{2, \alpha}(\mathbb S^4)$ (with a coordinate singularity at $t=1$, similar to the coordinate singularity
of polar coordinates at $r=0$) through
\[
w=v \circ \varphi^{-1}_{P,t} + \ln |d \varphi^{-1}_{P,t}|.
\]
 $w$ solves \eqref{2}  with $K$ replaced by $\Ks$ iff $v$ solves
\begin{equation}\label{treq}
\sigma_2(A_v)=\Ks \circ \varphi_{P,t} e^{4v}.
\end{equation}
Then the estimates for $w$ in {\bf Theorem \ref{chy2}} turn into the following estimates for $v$ and $t$.
\begin{proposition}
Assume that $K$ is a positive $C^2$ function on $\mathbb S^4$ satisfying the non-degeneracy condition \eqref{nd},
and let $w$ be a solution to \eqref{2} with $K$ replaced by $\Ks$ and $(v, P, t)\in \mathcal S_0\times \mathbb S^4
\times [1, \infty)$ be the coordinates of
$w$ defined in the paragraph above. Then there exist $t_0$ and  $\epsilon (s)>0$ with $\lim_{s\to 0} \epsilon (s) =0$, such that
\begin{equation}\label{pro1bd}
t \le t_0 \quad \text{ and } \quad ||v||_{C^{2, \alpha}(\mathbb S^4)} < \epsilon (s).
\end{equation}
\end{proposition}
A proof for {\bf Proposition 1} will be postponed to the end of the next section.

We treat \eqref{treq} as a  nonlinear map
\[
\Fs[v, \xi] := e^{-4 v(x)} \sigma_2(g_c^{-1} \circ \left[A_{g_c} -  \grad^2v+dv \otimes dv - 
\frac 12 |\grad v|^2 g_c   \right])-\Ks \circ \varphi_{P,t}, 
\]
from $\mathcal S_0 \times B$ into $C^{\alpha}(\mathbb S^4)$, for $0<s\le 1$, where  $\xi =(t-1)P/t \in B$.
{\bf Proposition 1} implies that there is a neighborhood $\mathcal N \subset \mathcal S_0$ of $0 \in  \mathcal S_0$
and $0<r_0=(t_0-1)/t_0<1$ such that $\Fs$ does not have a zero on 
$\partial ( \mathcal N \times B_{r})$ for all $r_0\le r <1$ and $0<s\le 1$. 
According to \cite{YL89}, there is a well defined degree for $\Fs$  on $\mathcal N \times B_{r_0}$ and 
it is independent of $0<s\le 1$. We will compute this degree of $\Fs$, for $s>0$ small, through the 
degree of a finite dimensional map.

We first use the implicit function theorem to define this map and link the solutions to \eqref{treq} to the zeros
of this map. 
Note that $F^{[0]}[0, (t-1)P/t] =0$ and $D_{v}F^{[0]}[0, (t-1)P/t](\eta)= -6 \Delta \eta - 24 \eta$.
If $\Pi$ denotes the projection from $C^{\alpha}(\mathbb S^4)$ into 
\[
Y:= \{ f \in C^{\alpha}(\mathbb S^4): \int_{\mathbb S^4} f x_j \, \vo =0\quad \text{for $j=1,
\cdots, 5$}\}
\]
defined by $\Pi(f)= f - 5|\mathbb S^4|^{-1} \sum_{j=1}^5 \left( \int_{\mathbb S^4} f x_j \, \vo \right) x_j$,
then we can apply the implicit function theorem to $\Pi \circ \Fs$ at $v=0$ to conclude
\begin{proposition}
There exist 
some neighborhood $\mathcal N_{\epsilon} \subset \mathcal N$ of $0\in \mathcal S_0$ and  $s_0>0$,
such that for all $0<s<s_0$, $(P,t) \in \mathbb S^4 \times [1,t_0]$, 
there exists a \emph{unique} $v=v(x;P,t,s) \in \mathcal N_{\epsilon}$, depending differentiably on
$(P,t,s)$ such that 
\begin{equation}\label{pro}
\Pi\circ \Fs [v(x;P,t,s), (t-1)P/t] = 0.
\end{equation}
Furthermore, there exists some $C>0$ such that, for $0<s\le s_0$, $1\le t\le t_0$, 
\begin{equation}\label{im}
||v(x;P,t,s)||_{C^{2,\alpha}(\mathbb S^4)} \le C ||\Ks \circ \varphi_{P,t} -6||_{C^{\alpha}(\mathbb S^4)}
=C s ||K\circ \varphi_{P,t}-6||_{C^{\alpha}(\mathbb S^4)}.
\end{equation}
\end{proposition}
\eqref{pro} implies that
\[
 \Fs [v(x;P,t,s), (t-1)P/t] =\sum_{j=1}^5  \Lambda_j(P,t,s) x_j,
\]
for some Lagrange multipliers $\Lambda_j(P,t,s)$, which depend  differentiably on $(P,t,s)$.
Or, equivalently, 
\begin{equation}\label{lag}
\sigma_2(A_{v(x;P,t,s)}) = \left( \Ks \circ \varphi_{P,t} + \sum_{j=1}^5  \Lambda_j(P,t,s) x_j\right) e^{4v(x;P,t,s)}.
\end{equation}
A zero of the map $\Ls (P,t) :=(  \Lambda_1 (P,t,s), \cdots, \Lambda_5 (P,t,s))$ corresponds to 
a solution to \eqref{treq}. {\bf Propositions 1} and {\bf 2}
 say that, for  $s_0>0$ small, all solutions $v\in \mathcal S_0$ to \eqref{treq}, for $0<s\le s_0$,
are in $\mathcal N_{\epsilon}$, thus correspond to the zeros of the map  $\Ls (P,t) $.

\begin{remark}
$||K\circ \varphi_{P,t}-6||_{C^{\alpha}(\mathbb S^4)}$ could become unbounded when $t\to \infty$;
yet thanks to the bound $1\le t\le t_0$ from \emph{\textbf{Proposition 1}}, it remains bounded 
in terms of $||K||_{C^{\alpha}(\mathbb S^4)}$ in the range
$1\le t\le t_0$.  Note also that $||K\circ \varphi_{P,t}-6||_{L^{p}(\mathbb S^4)}$ remains bounded
in terms of $||K||_{L^{\infty}(\mathbb S^4)}$  even in the range $1\le t < \infty$. It is essentially this bound
and the applicability of $W^{2,p}$ estimates in the semilinear setting of \cite{CGY93} and \cite{YL95} which
allowed them to handle their cases without using the bound $1\le t\le t_0$.
\end{remark}
\begin{remark}
The implicit function theorem procedure here  works also in the setting of
\cite{CY91}, \cite{CGY93} and \cite{YL95} using $W^{2,p}$ space,
as does  Proposition 1  in the setting of \cite{CGY93} and \cite{YL95},
and can be used to simplify the arguments there.
\end{remark}

At this point, we need the following Kazdan-Warner type identity for solutions to \eqref{2}.
\begin{proposition}\label{prop:KW}
Let $w$ be a solution to \eqref{2}. Then, for $1\le j \le 5$,
\begin{equation}\label{kw}
\int_{\mathbb S^4} \langle \grad K(x), \grad x_j \rangle e^{4w(x)} \,\vo =0.
\end{equation}
\end{proposition}
{\bf Proposition~\ref{prop:KW}} is a special case of the results in \cite{V00} and \cite{Han06}. But in the special case
of $\mathbb S^4$, it is a direct consequence of the  variational characterization of the solution to \eqref{2},
as given after the statement of Theorem~\ref{chy}. A solution $w$ to \eqref{2}
is a critical point of $F[w]= Y[w]-I\hskip -1mm I[w] +C_K[w]$ there, thus satisfies,
for any one-parameter family of conformal diffeomorphisms $\varphi_s$ of $\mathbb S^4$ with $\phi_0=$Id,
\[
\frac{d}{ds}\Big|_{s=0} F[w_{\varphi_s}]=0,
\]
with $w_{\varphi_s}= w \circ \varphi_s +  \log |d \varphi_s|$. Since
 $ Y[w_{\varphi_s}] = Y[w]$
and $ I\hskip -1mm I[w_{\varphi_s}]= I\hskip -1mm I[w]$,
 a solution $w$ of \eqref{2} thus satisfies
\[
\frac{d}{ds}\Big|_{s=0} C_K[w_{\varphi_s}]=\frac{d}{ds}\Big|_{s=0}  
\left(\mint K \circ \varphi_s^{-1}  e^{4w}\,  d \/ vol_{g_c} \right) =0,
\]
which is \eqref{kw}.

Applying \eqref{kw} to $v(x;P,t,s)$, a solution to \eqref{lag}, we obtain, 
\[
\int_{\mathbb S^4} \langle \grad \left(\Ks \circ \varphi_{P,t} (x)+\sum_{j=1}^5  \Lambda_j(P,t,s) x_j\right),  
\grad x_k \rangle e^{4 v(x;P,t,s)} \,\vo =0, 
\quad \text{for $1\le k \le 5$,}
\]
from which we obtain, for $1\le k \le 5$,
\[
\begin{split}
&-\sum_{j=1}^5  \Lambda_j(P,t,s) \int_{\mathbb S^4} \langle \grad  x_j, \grad x_k \rangle e^{4 v(x;P,t,s)} \,\vo \\
=& \int_{\mathbb S^4} \langle \grad \left(\Ks \circ \varphi_{P,t} (x) \right), \grad x_k \rangle e^{4 v(x;P,t,s)} \,\vo\\
=& s\int_{\mathbb S^4} \langle \grad \left(K \circ \varphi_{P,t} (x) \right), \grad x_k \rangle e^{4 v(x;P,t,s)} \,\vo.
\end{split}
\]
As in \cite{CY91}, \cite{CGY93} and \cite{YL95}, we define
\[
\As (P,t)= (4|\mathbb S^4|)^{-1} \int_{\mathbb S^4} \langle \grad \left(K \circ \varphi_{P,t} (x) \right),
\grad x \rangle e^{4 v(x;P,t,s)} \,\vo\in \mathbb R^5.
\]
Since 
\[
 \left(\int_{\mathbb S^4} \langle \grad  x_j, \grad x_k \rangle e^{4 v(x;P,t,s)} \,\vo \right)
\] 
is positive definite, we conclude that
\[
\deg(\As, B_{r_0}, O)=-\deg(\Ls, B_{r_0}, O),
\]
for $s_0>s>0$ provided that one of them is well defined.

Using $ v(x;P,t,s) \in \mathcal S_0$, and $\Delta x=-4 x$ on $\mathbb S^4$, we have, as in \cite{CY91},
\cite{CGY93} and \cite{YL95},
\[
\As (P,t)= G(P,t) + I + \Pi,
\]
where
\[
I= |\mathbb S^4|^{-1}  \int_{\mathbb S^4} \left(K \circ \varphi_{P,t} (x) - K(P) \right)
x \left( e^{4  v(x;P,t,s)} -1 \right)  \,\vo,
\]
and
\[
\Pi= - (4|\mathbb S^4|)^{-1} \int_{\mathbb S^4} \left(K \circ \varphi_{P,t} (x) - K(P) \right)
 \langle \grad x,  \grad  e^{4  v(x;P,t,s)} \rangle  \,\vo.
\]
We could have fixed $t_0\ge t_1$ such that $ G(P,t) \ne 0$ for $t = t_0$, and 
there will be a $\delta>0$ such that
$| G(P,t)| \ge \delta $ for $t=t_0$. 
Since \eqref{im} implies that
\[
||v(x;P,t,s)||_{C^{2,\alpha}(\mathbb S^4)} = O(s), \quad \text{uniformly for $(P,t)\in \mathbb S^4 \times [1, t_0]$,}
\]
we find that, by fixing $s_0>0$ small if necessary,
\[
|I|+|\Pi| \le \frac{1}{2} |G(P,t)|, \quad \text{for $0<s\le s_0$ and $t=t_0$.}
\]
This implies that $\As (P,t) \cdot G(P,t) >0$ for $0<s\le s_0$ and $t=t_0$. Therefore
\[
-\deg ( \Ls, B_{r_0}, O)=\deg ( \As, B_{r_0}, O) = \deg (G, B_{r_0}, O)\ne 0,
\]
 for $0<s\le s_0$. Finally, we now prove
\begin{equation}\label{ddf}
\deg(\Fs, \mathcal N\times B_{r_0}, O)
=- \deg ( \Ls, B_{r_0}, O),
\end{equation}
for $0<s\le s_0$, 
from which  follows  the existence of a solution to \eqref{2}.

The verification of \eqref{ddf} is routine, but requires several steps. 
First, we may perturb $K$, if necessary, within
 the class of functions satisfying the conditions in {\bf Theorem \ref{exist}} such that
the corresponding $G(P,t)$ has only  isolated and non-degenerate zeros in $B_{r_0}(O)$.
We will prove momentarily that for $s>0$ small, the zeros of $\Ls$ for  $s>0$ small will be close to
the zeros of $G(P,t)$ and are isolated, non-degenerate. Therefore
the zeros of $\Fs$ in $\mathcal N\times B_{r_0}$ are isolated and non-degenerate.
This can be argued as follows. First, it follows from \eqref{lag}  that
\[
\Ls (\xi) \cdot x =  (\text{Id}-\Pi) \left( e^{-4v(x;\xi, s)}\sigma_2(A_{v(x;\xi, s)})-\Ks \circ \varphi_{P,t}\right).
\]
Using \eqref{im}, we can then write
\[
e^{-4v(x;\xi, s)}\sigma_2(A_{v(x;\xi, s)})= 6-6\Delta v(x;\xi, s) - 24 v(x;\xi, s) + Q(v(x;\xi, s)),
\]
with $|| Q(v(x;\xi, s))||_Y\lesssim ||v(x;\xi, s)||^2_X \lesssim s^2$. Therefore,
using $(\text{Id}-\Pi)(1)=(\text{Id}-\Pi)\left(6\Delta v(x;\xi, s) + 24 v(x;\xi, s)\right)=0$,
and
\[
\Ls (\xi) = 5|\mathbb S^4|^{-1} \int_{\mathbb S^4} \left(\Ls (\xi) \cdot x\right) x \,\vo,
\]
we have
\[
\Ls (\xi) = -5s G(P,t) + 5 |\mathbb S^4|^{-1} \int_{\mathbb S^4} \left[  (\text{Id}-\Pi) \left(  Q(v(x;\xi, s)) \right)
\right] x \,\vo,
\]
with $|(\text{Id}-\Pi) \left(  Q(v(x;\xi, s)) \right)|\lesssim s^2$, so the zeros of $\Ls (\xi)$ for $s>0$ small
are close to the zeros of $G(P,t)$. We can further use the implicit function theorem to prove that
for $s>0$ small there is a (unique)  non-degenerate zero of $\Ls (\xi)$ near each zero of  $G(P,t)$.

\begin{remark} This argument shows that, for each non-degenerate zero of $G(P,t)$, if we associate $(P,t)$
with the center of mass of $\varphi_{P,t}$,
\[
C.M (\varphi_{P,t}) := |\mathbb S^4|^{-1} \int_{\mathbb S^4} \varphi_{P,t}(x)\, \vo \in B_1(O)
\]
as a geometric representation of $(P,t)$, then for $s>0$ small, there is a unique solution $w$ to \eqref{2}
whose center of mass approaches $C.M (\varphi_{P,t})$, for our argument gives rise to a solution
\[
w(x) = v(\cdot ;P',t')\circ \varphi_{P',t'}^{-1}(x) + \ln |d  \varphi_{P',t'}^{-1}(x)|
\]
with $(P',t')$ approaching $(P,t)$, and $ v(x;P',t')$ approaching $0$ as $s\to 0$, thus the center of mass
of $w$ is 
\[
|\mathbb S^4|^{-1} \int_{\mathbb S^4} e^{4w(x)}x\,\vo = |\mathbb S^4|^{-1} \int_{\mathbb S^4} e^{4v(y)} 
\varphi_{P',t'}(y) \,\vo \to C.M (\varphi_{P,t})
\]
as $s \to 0$.
\end{remark}
Now $\deg(\Fs, \mathcal N\times B_{r_0}, O)$ is well
defined in the manner of \cite{YL89}, and according to Propositions 2.1--2.4 of \cite{YL89}, 
\[
\deg(\Fs, \mathcal N\times B_{r_0}, O) =\sum_{\xi\in B_{r_0}(O): \Ls (\xi)=0} \ind (D\Fs [v(x;\xi, s), \xi]),
\]
where  $\ind (D\Fs [v(x;\xi, s), \xi])$ refers to the index of the linear operator
$D\Fs [v(x;\xi, s), \xi]$, and is computed as $(-1)^{\beta}$, with $\beta$ denoting the number of
negative eigenvalues of $D\Fs [v(x;\xi, s), \xi]$. We also have
\[
\deg(\Ls, B_{r_0}, O) =\sum_{\xi\in B_{r_0}(O): \Ls (\xi)=0} \ind (D\Ls(\xi)).
\]

To compute  $D\Fs [v(x;\xi, s), \xi]$, we identify $\xi \in \mathbb R^5$ with $\xi \cdot x \in \text{span}\{x_1,
\cdots, x_5\}$, and write the differential of $\Fs$ in the direction of $\dv$ as
$D_v \Fs [v(x;\xi, s), \xi](\dv)$, or simply $D_v \Fs (\dv)$, and
the differential of $\Fs$ in the direction of $\dxi \cdot x$ as $D_{\xi} \Fs [v(x;\xi, s), \xi](\dxi)$. 
Then
\[
D_{\xi} \Fs [v(x;\xi, s), \xi](\dxi) =-s \dxi \cdot \grad_{\xi} \left(K \circ \varphi_{P,t}\right),
\]
and
\[
D_v \Fs [v(x;\xi, s), \xi](\dv)= M^{ij}[v(x;\xi, s)]\grad^{v(x;\xi, s)}_{ij} \dv -4 \Ks \circ \varphi_{P,t} \dv,
\]
where $ M^{ij}[v(x;\xi, s)]$ stands for the Newton tensor associated with $\sigma_2( e^{-2v(x;\xi, s)}A_{v(x;\xi, s)})$,
and $\grad^{v(x;\xi, s)}_{ij}$ stands for the covariant differentiation in the metric $e^{2v(x;\xi, s)}g_c$.
Thus,
\[
D \Fs [v(x;\xi, s), \xi]( \dv + \dxi \cdot x)=  M^{ij}[v(x;\xi, s)]\grad^{v(x;\xi, s)}_{ij} \dv -
4 \Ks \circ \varphi_{P,t} \dv- s \dxi \cdot \grad_{\xi} \left(K \circ \varphi_{P,t}\right).
\]
At a fixed zero $\xi$ of $\Ls (\xi)=0$,  we 
define a family of deformed linear operators $L_{\tau, s}$ for $0\le \tau \le 1$ by
\[
L_{\tau, s}( \dv + \dxi \cdot x)=  M^{ij}[v^{[\tau]}] \grad^{v^{[\tau]}}_{ij}\dv -4K^{[s\tau]} \circ \varphi_{P,t}\dv - 
s \dxi \cdot \grad_{\xi} \left(K \circ \varphi_{P,t}\right),
\]
where $v^{[\tau]}=\tau v(x;\xi, s)$, and 
$\dv \in  X:=\{ \dv\in C^{2,\alpha}(\mathbb S^4): \int_{\mathbb S^4} \dv (x) x_j =0, j=1, \cdots, 5\}$. 
Then $L_{\tau, s}$ defines self-adjoint operators with respect to the
metric $e^{2v^{[\tau]}}g_c$, thus its eigenvalues are all real. We first assume the

\begin{claim} For $s>0$ small and 
$0\le \tau \le 1$, the spectrum of $L_{\tau, s}$ does not contain zero.
\end{claim}
Thus $\ind (L_{0,s})=\ind(L_{1,s})
=\ind (D \Fs [v(x;\xi, s), \xi])$. We will next establish 
\begin{equation}\label{index}
\ind (L_{0,s})=(-1)^{1 +  \gamma}, \quad \text{ for $s>0$ small},
\end{equation}
where $\gamma$ is the number of positive eigenvalues of $\grad G(P,t)$ at $\xi$, and
\begin{equation}\label{index2}
\gamma = \text{ the number of negative eigenvalues of $D_{\xi} \Ls (\xi)$ for  $s>0$ small}.
\end{equation}
First note that
\[
L_{0,s}( \dv + \dxi \cdot x)= -6\Delta \dv -24 \dv -s \dxi \cdot \grad_{\xi} \left(K \circ \varphi_{P,t}\right),
\]
so if $\dv + \dxi \cdot x$ is an eigenfunction corresponding to a negative eigenvalue $-\lambda$,
with $\dv \in  X$, then
\[
-6\Delta \dv -24 \dv -s \dxi \cdot \grad_{\xi} \left(K \circ \varphi_{P,t}\right)=-\lambda (\dv + \dxi \cdot x).
\]
Taking projection in $\text{span}\{x_1,\cdots, x_5\}$, we find
\[
-s \grad G(P,t)\dxi =-\frac{\lambda}{5} \dxi,
\]
and taking projection in $X$, we find
\begin{equation}\label{proj2}
-6\Delta \dv -24 \dv -s \Pi \left(\dxi \cdot \grad_{\xi} \left(K \circ \varphi_{P,t}\right)\right)= -\lambda \dv.
\end{equation}
If $\dxi \neq 0$, then $\dxi$ is an eigenvector of $\grad G(P,t)$ with eigenvalue $\frac{\lambda}{5s} >0$;
and if  $\dxi =0$, then $\dv \neq 0$ solves $-6\Delta \dv -24 \dv = -\lambda \dv$, which is possible for 
some $\lambda>0$ iff $-\lambda =-24$ and $\dv = \text{constant}$. Conversely, for any
eigenvector $\dxi \neq 0$ of $\grad G(P,t)$ with eigenvalue $\mu>0$,
the operator $-6\Delta  -24 + 5s\mu$ is an isomorphism from $X$ to $Y$ for $s>0$ small, so we can
solve \eqref{proj2}
as $(-6\Delta  -24 + 5s\mu) \dv = s \Pi \left(\dxi \cdot \grad_{\xi} \left(K \circ \varphi_{P,t}\right)\right)$
for $\dv \in X$ and $\dv + \dxi \cdot x$ becomes an eigenfunction of $L_{0,s}$ with
eigenvalue $- 5s\mu$.  Therefore we conclude \eqref{index}.

We now  establish \eqref{index2} to prove  $\ind (D \Fs [v(x;\xi, s), \xi])=-\ind (D_{\xi}\Ls (\xi))$
for $s>0$ small. 
From \eqref{lag}, which can be written as $\Ls (\xi) \cdot x = \Fs [v(x;\xi, s), \xi])$, we obtain
\[
D_v \Fs (D_{\xi} v(x; \xi, s) (\dxi)) + D_{\xi} \Fs (\dxi) = D_{\xi}\Ls (\xi) (\dxi) \cdot x.
\]
Taking projections in $X$ and $\text{span}\{x_1,\cdots, x_5\}$, respectively, and using 
$D_{\xi} \Fs (\dxi) =-s \dxi \cdot \grad_{\xi} \left(K \circ \varphi_{P,t}\right)$,
we obtain
\begin{equation}\label{pro1}
\Pi \left(D_v \Fs (D_{\xi} v(x; \xi, s) (\dxi))\right) -s \Pi \left( \dxi \cdot \grad_{\xi} \left(K\circ \varphi_{P,t}
\right)\right)=0,
\end{equation}
and
\begin{equation}\label{pro2}
\mint ( \text{Id}-\Pi)  \left(D_v \Fs (D_{\xi} v(x; \xi, s) (\dxi))\right) x\,\vo  -s \grad_{\xi} G(P,t)\dxi
 = \frac{1}{5} D_{\xi}\Ls (\xi) (\dxi).
\end{equation}
Writing
\[
D_v \Fs (D_{\xi} v(x; \xi, s) (\dxi))= (-6\Delta -24)(D_{\xi} v(x; \xi, s)(\dxi)) +\varTheta(D_{\xi} v(x; \xi, s)(\dxi)),
\]
we find, using \eqref{im}, that $||\varTheta(D_{\xi} v(x; \xi, s)(\dxi))||_Y\lesssim s ||D_{\xi} v(x; \xi, s)(\dxi)||_X$.
Thus $\Pi \left(D_v \Fs (\cdot)\right): X\mapsto Y$ is an isomorphism for $s>0$ small and has
an inverse $\Psi$, and we can solve
$D_{\xi} v(x; \xi, s) (\dxi)$ in terms of $\dxi$ from \eqref{pro1}:
\[
D_{\xi} v(x; \xi, s) (\dxi) = \Psi \left( s \Pi \left( \dxi \cdot \grad_{\xi} \left(K\circ \varphi_{P,t} 
\right)\right) \right) := s \Upsilon (\dxi).
\]
Using this in \eqref{pro2}, we find
\[
\begin{split}
\frac{1}{5} D_{\xi}\Ls (\xi) 
&= -s \grad_{\xi} G(P,t) + s \mint \left[(\text{Id}-\Pi) \circ \varTheta \circ \Upsilon \right] x\, \vo\\
&= -s \left(\grad_{\xi} G(P,t) + O(s)\right).\\
\end{split}
\]
Thus for $s>0$ small, $\gamma$ matches the number of negative eigenvalues of $ D_{\xi}\Ls (\xi)$, and
we can conclude that $\ind(D\Fs( v(x; \xi, s),\xi))=-\ind(D_{\xi}\Ls (\xi))$.

In the remainder of this section, we provide proof for our {\bf Claim} above, 
leaving the proof for  {\bf Proposition 1} to the end of the next section.
\begin{proof}[Proof of \emph{\textbf{Claim}}] 
Suppose that for (a sequence of) $s>0$ small and some $0\le \tau \le 1$, $L_{\tau, s}$ has $\dv +\dxi \cdot x$, with
$\dv \in X$, $\dxi \in \mathbb R^5$, as eigenfunction with zero eigenvalue.
Then, taking projections in $\text{span}\{x_1, \cdots, x_5\}$ and $X$, respectively, we obtain
\begin{equation}\label{pro1'}
\Pi\left[ M^{ij}[v^{[\tau]}]\grad^{v^{[\tau]}}_{ij} \dv - 4 K^{[s\tau]} \circ \varphi_{P,t} \dv\right] -s \Pi\left[
\dxi \cdot \grad_{\xi}\left(K \circ \varphi_{P,t}\right)\right]=0,
\end{equation}
and
\begin{equation}\label{pro2'}
\mint (\text{Id}-\Pi) \left[ M^{ij}[v^{[\tau]}]\grad^{v^{[\tau]}}_{ij} \dv - 4 K^{[s\tau]} \circ \varphi_{P,t} \dv\right]
x\,\vo  -s \grad_{\xi} G(P,t) \dxi =0.
\end{equation}
Using \eqref{im} again, we find
\[
 M^{ij}[v^{[\tau]}]\grad^{v^{[\tau]}}_{ij} \dv - 4 K^{[s\tau]} \circ \varphi_{P,t} \dv=
-6 \Delta  \dv - 24 \dv + \varTheta^{\tau,s}(\dv),
\]
with $||\varTheta^{\tau,s}(\dv)||_Y\lesssim s\tau ||\dv||_X$.
Thus, for $s>0$ small,  we can solve $\dv$ from \eqref{pro1'} to obtain
\[
\dv = \Psi \left(s  \Pi\left[ \dxi \cdot \grad_{\xi}\left(K \circ \varphi_{P,t}\right) \right]\right)
=s \Upsilon (\dxi).
\]
Thus $\dxi \ne 0$ and we can normalize it so that $|\dxi|=1$. Using this in \eqref{pro2'}, we find
\[
s(\text{Id}-\Pi) \circ  \varTheta^{\tau,s} \circ \Upsilon (\dxi) - s \grad_{\xi} G(P,t) \dxi =0.
\]
Using $||\varTheta^{\tau,s}||\lesssim s\tau$, we find this impossible for $s>0$ small under our non-degeneracy
assumption on the zeros of $G(P,t)$.
\end{proof}

\vskip12pt
\section{Proof of \eqref{ptws'}, \eqref{grad4}, Theorem \ref{chy2} and Proposition 1}
\vskip10pt

\begin{proof}[Proof of \eqref{ptws'} and  \eqref{grad4}]
The full strength of \eqref{ptws} is established as soon as the $W^{2,3}$ estimates are established ---
the latter is a step in proving the $W^{2,6}$ estimates.

If there is a sequence of solutions $w_j$ to \eqref{2} such that $\max w_j = w_j (P_j) \to \infty$, then we choose
 conformal automorphism $\phi_j=\phi_{P_j, t_j}$ of  $\mathbb S^4$, such that the rescaled function
 \begin{equation} \label{resca}
 v_j (P) = w_j \circ \phi_j (P) + \ln | d \phi_j(P)|,
 \end{equation}
 satisfies the normalization condition
 \begin{equation} \label{norm}
 v_j (P_j) = \frac 14 \ln \frac {6}{K(P_j)}.
 \end{equation}
 If we use stereographic  coordinates for $\mathbb S^4$, with $P_j$ as the north pole, then
 \[
 y\left(\phi_j(P)\right)= t_j y(P), \mbox{ for } P \in \mathbb S^4,
 \]
 and
 \[
 v_j (P) = w_j \circ \phi_j (P)  + \ln \frac{t_j\left(1+|y(P)|^2\right)}{1+t_j^2|y(P)|^2}.
 \]
 $ v_j$ would satisfy
 \begin{equation} \label{norE}
 \sigma_2(A_{v_j}) = K \circ \phi_j e^{4v_j}.
 \end{equation}
 The normalization in \eqref{norm} amounts to choosing $t_j$ such that
 \[
 w_j(P_j)-\ln t_j = \frac 14 \ln \frac {6}{K(P_j)}.
 \]
 Thus, $t_j \to \infty$, and for any $P \in \mathbb S^4$,
 \begin{equation} \label{upper}
 v_j (P) \leq \frac 14 \ln \frac {6}{K(P_j)} +  \ln \frac{t_j^2\left(1+|y(P)|^2\right)}{1+t_j^2|y(P)|^2}.
 \end{equation}
 \eqref{upper} implies that, away from $-P_j$, $v_j$ has an upper bound independent of $j$.
 Together with \eqref{norE}, the local gradient and higher derivative estimates of \cite{GW1}, there exists
 a subsequence, still denoted as $\{v_j\}$, such that, $P_j \to P_*$, and for any $\delta >0$,
 \begin{equation} \label{ptw}
 v_j \to v_{\infty} \mbox{ in } C^{2,\alpha}\left( \mathbb S^4 \setminus B_{\delta}(-P_*)\right),
 \quad \text{ for some limit } \quad  v_{\infty}.
 \end{equation}
We also have
\begin{gather}
\sigma_2( A_{v_{\infty}} ) = K(P_*) e^{4 v_{\infty}} \qquad \text{ on } \quad \mathbb S^4 \setminus \{-P_*\},\\
\int_{ \mathbb S^4}  K(P_*) e^{4 v_{\infty}}\, \vo  \leq \liminf_{j \to \infty} \int_{ \mathbb S^4}
K \circ \phi_j e^{4v_j}\, \vo = 16 \pi^2, \\
v_{\infty}(P_*)= \frac 14 \ln \frac {6}{K(P_*)}, \quad \grad  v_{\infty}(P_*) =0, \label{north}\\
v_{\infty}(P) \le \frac 14 \ln \frac {6}{K(P_*)}+ \ln \frac{1+|y(P)|^2}{|y(P)|^2}.
\end{gather}

A Liouville type classification result in \cite{CGY02b} and \cite{LL03} says that 
$$v_{\infty}-\frac 14 \ln \frac {6}{K(P_*)}= \ln |d\phi|$$
 for some conformal automorphism $\phi$ of $\mathbb S^4$, which together with
\eqref{north} implies that
\begin{equation} \label{limit}
v_{\infty} \equiv \frac 14 \ln \frac {6}{K(P_*)}.
\end{equation}
Thus for any $\delta >0$,
\[
\lim_{j \to \infty} \int_{ \mathbb S^4 \setminus B_{\delta}(-P_*)} K \circ \phi_j e^{4v_j}\, \vo
= 6 \left| \mathbb S^4 \setminus B_{\delta}(-P_*) \right|.
\]
Together with the Gauss-Bonnet formula
\[
\int_{ \mathbb S^4}  K \circ \phi_j e^{4v_j}\, \vo = 6 \left| \mathbb S^4\right|,
\]
we have
\[
\lim_{j \to \infty} \int_{  B_{\delta}(-P_*)} K \circ \phi_j e^{4v_j}\, \vo = 6 \left|  B_{\delta}(-P_*) \right|.
\]
This allows us to apply our Theorem in \cite{Han04} on
$ B_{\delta}(-P_*)$ for small $\delta >0$ to conclude that $ \exists C >0$, such that
\begin{equation} \label{pupper}
\max_{\mathbb S^4} v_j \le C.
\end{equation}

Next we declare the
\begin{claim}
There exists $C'>0$ such that
\begin{equation} \label{lower}
\min_{\mathbb S^4}  v_j \ge - C'.
\end{equation}
\end{claim}
The Claim can be proved making use of the information that $R_{v_j} = R_{w_j} \circ \phi_j \ge 0$, 
which implies
\begin{equation} \label{scar}
2 - \Delta v_j - |\grad v_j |^2 \ge 0.
\end{equation}
Thus
\begin{equation} \label{ave}
\begin{split}
v_j(P)- \bar{v_j} &= \int_{\mathbb S^4} \left( - \Delta v_j (Q) \right) G(P,Q)\, \vo (Q) \\
	& \ge -2 \int_{\mathbb S^4}  G(P,Q)\, \vo (Q) ,
	\end{split}
	\end{equation}
	where $ G(P,Q)$ is the Green's function of $- \Delta$ on $\mathbb S^4$.
Integrating \eqref{scar} over $ \mathbb S^4$ implies that
\begin{equation} \label{gave}
2 \ge \mint | \grad v_j |^2 \ge \mbox{const. } \left( \mint | v_j(P)- \bar{v_j} |^4 \right)^{\frac 12} .
\end{equation}
\eqref{ptw}, \eqref{limit}, \eqref{ave}, and \eqref{gave} conclude the Claim and \eqref{ptws'}.

Next we prove the  integral estimate \eqref{grad4}.
This can be seen by looking at the integral version of the equation
\[
\int_{\mathbb S^4} \left[2|\grad v_j|^2 + 2 \Delta v_j -6 \right] \langle \grad v_j , \grad \eta \rangle
+ \Delta \eta |\grad v_j|^2 + \left( K \circ \phi_j e^{4 v_j} -6 \right) \eta =0.
\]
If we plug in $\eta = v_j$, we obtain
\[
\int_{\mathbb S^4} \left( 6 - 2 |\grad v_j|^2 - 3 \Delta v_j \right) |\grad v_j|^2 = 
\int_{\mathbb S^4} \left(  K \circ \phi_j e^{4 v_j} -6 \right) v_j.
\]
Using $\Delta v_j \le 2 - |\grad v_j|^2 $, we have
\begin{equation}\label{inegrad}
\int_{\mathbb S^4} |\grad v_j |^4 \le \int_{\mathbb S^4} \left(  K \circ \phi_j e^{4 v_j} -6 \right) v_j,
\end{equation}
which converges to $0$ by \eqref{ptw}, \eqref{limit}, \eqref{pupper}, \eqref{lower}
and the Dominated Convergence Theorem.
\end{proof}

Next, we prove {\bf Theorem \ref{chy2}}. We will first prove that, under our non-degeneracy conditions on $K$, there
is a bound $C>0$ depending on the quantities as in the statements of  {\bf Theorem  \ref{chy2}},
but uniform in $0<s\le 1$, such that any solution $w$ of \eqref{2} with $\Ks$ satisfies
$\max_{\mathbb S^4} w_j \le C$. Once we have  the  bound $\max_{\mathbb S^4} w_j\le C$, the $C^{2, \alpha}$
estimates follow from known theory of fully nonlinear elliptic equations.

\begin{proof}[Proof of  \emph{\bf Theorem \ref{chy2}}]
Suppose, on the contrary, that $\max_{\mathbb S^4} w_j \to \infty$ (for a sequence of $K$'s, which we 
write as a single $K$ for simplicity, satisfying the bounds in   {\bf Theorem \ref{chy2}}). 
Then, as proved above, \eqref{ptws'} holds.
Let $P_j, t_j$ be as defined in the earlier part of the proof.
We will then prove the following estimates:

\begin{equation} \label{1der}
|\grad K(P_j)| = \frac {o(1)}{t_j}, \quad \text{as } j \to \infty,
\end{equation}
and
\begin{equation} \label{2der}
\Delta K(P_j) \to 0, \quad \text{as } j \to \infty.
\end{equation}
\eqref{1der} and \eqref{2der} would contradict our hypotheses on $K$.

The main idea is to examine the Kazdan-Warner identity in the light of the asymptotic profile of $w_j$
as given in {\bf Theorem 1}.

For each $w_j$, we choose stereographic coordinates with $P_j$ as the north
pole. For $P=(x_1, \cdots, x_5) \in \mathbb S^4$, let its stereographic coordinates
be $y=(y_1, \cdots, y_4)$. Also set $x'=(x_1, \cdots, x_4)$. Then
\begin{equation}
\left\{
\begin{aligned}
x_i &= \frac {2y_i}{1+|y|^2}, \qquad i=1, 2, 3, 4, \\
x_5 &= \frac{|y|^2-1}{|y|^2+1}.
\end{aligned}
\right.
\end{equation}
For any $\epsilon >0$, there exists $M>0$ such that
for any $P$ with $|y(P)|>M$, we have

\begin{equation} \label{tay}
K(P) = K(P_j) + \sum_{i=1}^4 a_i x_i + \sum_{k,h=1}^4 b_{hk}x_h x_k + r(P),
\end{equation}
with
\begin{equation} \label{tail}
|r(P)| \le \ep |x'|^2, \quad |x'| |\grad r(P)| \le \ep  |x'|^2, \quad  |x'|^2 |\grad^2 r(Q)| \le \ep  |x'|^2.
\end{equation}
We can identify $a_i = \grad_i K(P_j)$, $b_{hk} = \grad_{hk}K(P_j)$, and we may assume that
$b_{hk}$ is diagonalized: $b_{hk} = \delta_{hk} b_h$.
Then, using
\[
\grad x_1 = (1-x_1^2, -x_1 x_2, \cdots, -x_1 x_5), \cdots, \grad x_5 =(- x_5 x_1, \cdots, -x_5 x_4, 1-x_5^2),
\]
and
\[
\grad x_i \cdot \grad x_h = \delta_{ih} - x_i x_h,
\]
we have, for $1\le h \le 4$,
\[
\langle \grad K, \grad x_h \rangle = a_h - \sum_{i=1}^4 a_i x_i x_h + 2b_h x_h
-2\sum_{i=1}^4 b_{i} x_i^2 x_h + \grad r \cdot \grad x_h.
\]
So we can fix $M$ large such that, when $|y|>M$, 

\begin{equation}
\left\{
\begin{aligned}
\langle \grad K, \grad x_h \rangle &= a_h + 2b_h x_h + r_1(P), \\
|r_1(P)| &\le \ep |x'|.
\end{aligned}
\right.
\end{equation}
From the Kazdan-Warner identity, we have

\[
\begin{split}
0&= \int_{\mathbb S^4} \langle \grad \Ks, \grad x_h \rangle e^{4w_j}\, \vo \\
&=s \int_{\mathbb S^4} \langle \grad K, \grad x_h \rangle e^{4w_j}\, \vo 
 \end{split}
 \]
Thus, the deformation parameter $s$ is divided out  from the Kazdan-Warner identity to give
\[
\begin{split}
0&= \int_{\mathbb S^4} \langle \grad K, \grad x_h \rangle e^{4w_j}\, \vo \\
 &= \int_{|y|\le M} \langle \grad K, \grad x_h \rangle e^{4w_j}\, \vo + \int_{|y| > M}
 \langle \grad K, \grad x_h \rangle e^{4w_j}\, \vo.
 \end{split}
 \]
\begin{remark}
It is this property that $K$, not $\Ks$, can be used in the Kazdan-Warner identity
that allows us to obtain bounds on $w$ uniform in $0<s\le 1$.
This also applies to the settings in  \cite{CGY93} and \cite{YL95} to make the
estimates there uniform in $0<s\le 1$ in the respective deformations.
\end{remark}
 We estimate
 \[
 \begin{split}
 \int_{|y|\le M} \langle \grad K, \grad x_h \rangle e^{4w_j}\, \vo
 &\le C  \int_{|y|\le M} e^{4w_j} \left(\frac{2}{1+|y|^2}\right)^4 d\,y \\
 &= C \int_{ |z| \le \frac{M}{t_j}} e^{4v_j} \left(\frac{2}{1+|z|^2}\right)^4 d\,z\\
 &\le C \left(  \frac{M}{t_j} \right)^4, \quad \text{using \eqref{pupper} and \eqref{lower}}.
 \end{split}
 \]
 \[ 
 \int_{|y| > M}  \langle \grad K, \grad x_h \rangle e^{4w_j}\, \vo
 = a_h \int_{|y| > M} e^{4w_j}\, \vo + 2 b_h \int_{|y| > M} e^{4w_j} x_h\, \vo
 + \int_{|y| > M} e^{4w_j} r_1(P)\, \vo.
 \]
 The following estimates will complete the proof of \eqref{1der}.
 \begin{gather}
 \lim_{j\to \infty} \int_{|y| > M} e^{4w_j}\, \vo = \frac{6}{K(P_*)} \left| \mathbb S^4 \right|.\label{e1} \\
\int_{|y| > M} e^{4w_j} x_h\, \vo = \frac {o(1)}{t_j}, \quad \text{as } j \to \infty \; (1\le h \le 4).\label{e2}\\
 \int_{|y| > M} e^{4w_j} r_1(P)\, \vo = \frac {o(1)}{t_j}, \quad \text{as } j \to \infty. \label{e3}
 \end{gather}

Here are the verifications of the above estimates.

\[
\int_{|y| > M} e^{4w_j}\, \vo = \int_{ |z| > \frac{M}{t_j}}  e^{4v_j} \left(\frac{2}{1+|z|^2}\right)^4 d\,z
\to \frac{6}{K(P_*)} \left| \mathbb S^4 \right|,
\]
by \eqref{ptw}, \eqref{pupper} and \eqref{lower}.
\[
\begin{split}
\int_{|y| > M} e^{4w_j} x_h\, \vo &= \int_{ |z| > \frac{M}{t_j}} \frac{2t_j z_h}{1+t_j^2|z|^2}
e^{4v_j} \left(\frac{2}{1+|z|^2}\right)^4 d\,z \\
&= \int_{ |z| > \frac{M}{t_j}} \frac{2t_j z_h}{1+t_j^2|z|^2}
\left(e^{4v_j} - \frac{6}{K(P_*)} \right) \left(\frac{2}{1+|z|^2}\right)^4 d\,z \\
&= \int_{ |z| > \delta} + \int_{ \delta > |z| > \frac{M}{t_j}},
\end{split}
\]
with
\[
\left| \int_{ \delta > |z| > \frac{M}{t_j}} \right| \le C \int_{ \delta > |z| > \frac{M}{t_j}}
	\frac{1}{t_j|z|} d\,z \le \frac{C\delta^3}{t_j}.
	\]
	For any given $\ep>0$, we can first fix $\delta >0$ such that $C\delta^3 < \ep$.
	Then using the convergence of $v_j$ to $\frac 14 \ln \frac{6}{K(P_*)}$ on $|z| > \delta$,
	we can fix $J$ such that when $j \ge J$, we have $\left| e^{4v_j}- \frac{6}{K(P_*)}\right| < \ep$.
	Then
	\[
	\left| \int_{ |z| > \delta} \right| \le \ep \int_{|z| > \delta} \frac{1}{t_j|z|} \left(\frac{2}{1+|z|^2}\right)^4 d\,z \le \frac{C\ep}{t_j}.
	\]
	These together prove the second estimate above. \eqref{e3} follows similarly.

Finally
\[
\langle \grad K, \grad x_5 \rangle= -\sum_{i=1}^4 a_i x_i x_5 - 2 \sum_{i=1}^4 b_i x_i^2 x_5 + 
\grad r \cdot \grad x_5.
\]
We may fix $M$ large so that $|\grad r \cdot \grad x_5 | \le \ep |x'|^3$ when $|y|>M$.
In
\begin{equation} \label{k5}
\begin{split}
0&= \int_{\mathbb S^4} \langle \grad K, \grad x_5 \rangle e^{4w_j}\, \vo \\
 &= \int_{|y|\le M} \langle \grad K, \grad x_5 \rangle e^{4w_j}\, \vo + \int_{|y| > M}
  \langle \grad K, \grad x_5 \rangle e^{4w_j}\, \vo,
   \end{split}
    \end{equation}
    \begin{equation} \label{ins}
    \left| \int_{|y|\le M} \langle \grad K, \grad x_5 \rangle e^{4w_j}\, \vo \right| \le C \left(
    \frac{M}{t_j}\right)^4,
    \end{equation}
    as before.
    \[
    \left| \int_{|y| > M} x_i x_5 e^{4w_j}\, \vo \right| = \frac{o(1)}{t_j}
    \]
    as in the proof of \eqref{e2}. Thus
    \begin{equation} \label{out1}
     \left| \int_{|y| > M} a_i x_i x_5 e^{4w_j}\, \vo \right|  = \frac{o(1)}{t_j^2}.
     \end{equation}

    \[
    \begin{split}
    &\int_{|y| > M} x_i^2 x_5 e^{4w_j}\, \vo \\
    = &\int_{|z|> \frac{M}{t_j}}\left( \frac{2t_j z_i}{1+t_j |z|^2}\right)^2 \frac{t_j |z|^2-1}{t_j |z|^2+1}
    e^{4v_j} \left( \frac{2}{1+|z|^2} \right)^4 d\,z \\
    = &\int_{|z|> \frac{M}{t_j}}\left( \frac{2t_j z_i}{1+t_j |z|^2}\right)^2 \frac{t_j |z|^2-1}{t_j |z|^2+1} 
      \left(e^{4v_j} - \frac{6}{K(P_*)} \right) \left( \frac{2}{1+|z|^2} \right)^4 d\,z \\
	&+ \frac{6}{K(P_*)}\int_{|z|> \frac{M}{t_j}} \left( \frac{2t_j z_i}{1+t_j |z|^2} \right)^2 \frac{t_j |z|^2-1}{t_j |z|^2+1}
	\left( \frac{2}{1+|z|^2} \right)^4 d\,z 
\end{split}
\]
Note that
\begin{equation} \label{out2}
\begin{split}
&\int_{|z|> \frac{M}{t_j}} \left( \frac{2t_j z_i}{1+t_j |z|^2} \right)^2 \frac{t_j|z|^2-1}{t_j |z|^2+1}
\left( \frac{2}{1+|z|^2} \right)^4 d\,z \\
\asymp &\frac{4}{t_j^2} \int_{|z|> \frac{M}{t_j}} \frac{z_i^2}{|z|^4} \left( \frac{2}{1+|z|^2} \right)^4 d\,z\\
\asymp &\frac{1}{t_j^2} \int_{|z|> \frac{M}{t_j}} \frac{1}{|z|^2} \left( \frac{2}{1+|z|^2} \right)^4 d\,z\\
\asymp &\frac{1}{t_j^2} \left( \int_0^{\infty} ( \frac{2}{1+r^2})^4 r d\,r \right) |\mathbb S^3|
\end{split}
\end{equation}
Similarly, we can prove
\begin{equation} \label{out2'}
\left| \int_{|z|> \frac{M}{t_j}}\left( \frac{2t_j z_i}{1+t_j |z|^2}\right)^2 \frac{t_j |z|^2-1}{t_j |z|^2+1}
      \left(e^{4v_j} - \frac{6}{K(P_*)} \right) \left( \frac{2}{1+|z|^2} \right)^4 d\,z \right|
      = \frac{o(1)}{t_j^2},
      \end{equation}
      and
\begin{equation} \label{out3}
\begin{split}
    &\int_{|y| > M} |x'|^3  e^{4w_j} \\
= & \int_{|z|> \frac{M}{t_j}} \left( \frac{2t_j |z|}{1+t_j^2 |z|^2} \right)^3 e^{4v_j}
    \left( \frac{2}{1+|z|^2} \right)^4 d\,z \\
    = &  \frac{ O(1) }{t_j^3}.
    \end{split}
\end{equation}
To put things together, we multiply \eqref{k5} by $t_j^2$ and use
     \eqref{ins}, \eqref{out1}, \eqref{out2}, and \eqref{out3} to see that
     \[
     0 = o(1) -2 \Delta K(P_j)  \left( |\mathbb S^3| \int_0^{\infty} ( \frac{2}{1+r^2})^4 r d\,r +o(1) \right) +
   \frac{o(1)}{t_j},
   \]
   which shows \eqref{2der}.

\end{proof}
\begin{proof}[Proof of \emph{\textbf{Proposition 1}}]
First, by {\bf Theorem \ref{chy2}}, there is a $C>0$ depending on $K$ and $0<\alpha <1$ such that any solution $w$ to 
\eqref{2} with $K$ substituted by $\Ks$ and $0<s\le 1$ satisfies
\begin{equation}\label{wbd}
||w||_{C^{2, \alpha}(\mathbb S^4)} < C.
\end{equation}
Since $v= w \circ \varphi_{P,t} + \ln |d  \varphi_{P,t}|$ is chosen such that
\[
\mint e^{4v(y)} y\, \vo =0,
\]
we obtain, in terms of $w$ and $(P,t)$,
\begin{equation}\label{ba}
0=\mint e^{4w(x)} \varphi_{P,t}^{-1}(x)\, \vo = \mint e^{4w(x)} \varphi_{P,t^{-1}} (x)\, \vo,
\end{equation}
Due to \eqref{wbd}, there is a $\delta >0$ such that 
\[
\mint e^{4w(x)} \ge \delta.
\]
If there existed a sequence of solutions $w_j$ for which $t_j \to\infty$, we would have, computing
in stereographic coordinates in which $P_j$ is placed at the north pole, $\varphi_{P_j,t_j^{-1}} (x) \to 
(0, \cdots, 0, -1)$ except at $x=P_j$, therefore, in view of \eqref{wbd}, 
\[
\mint e^{4w(x)} \varphi_{P,t^{-1}} (x)\,\vo \to (0, \cdots, 0, - \mint e^{4w(x)}) \ne 0,
\]
contradicting \eqref{ba} above.
This implies the existence of some $t_0$ such that $t\le t_0$. 
Using this and \eqref{wbd} in the relation between $w$ and $v$, we find an upper bound for
$||v||_{C^{2,\alpha}(\mathbb S^4)}$. Finally
using the equation for $v$:
\[
\sigma_2(A_v)= \Ks \circ  \varphi_{P,t} e^{4v},
\]
in which the right hand side has an upper bound in $C^{2,\alpha}(\mathbb S^4)$ due to
$C^{2,\alpha}(\mathbb S^4)$ estimates of $v$ and  the bound $t\le t_0$, we find higher derivative 
bounds for $v$. Then as $s\to 0$, a subsequence of $v$ would converge to a limit $v_{\infty}$
in $C^{2,\alpha}(\mathbb S^4)$, which satisfies
\[
\sigma_2(A_{v_{\infty}})= 6  e^{4 v_{\infty}} \quad \text{and} \quad
\mint  e^{4 v_{\infty}(x)} x\,\vo =0.
\]
This implies that $v_{\infty}\equiv 0$. Since this limit $v_{\infty}$ is unique, we obtain
that $v \to 0$ in $C^{2,\alpha}(\mathbb S^4)$ as $s \to 0$, which is the remaining part of \eqref{pro1bd}.
\end{proof}

\vskip12pt
\section{Proof of the $W^{2,6}$ estimates of \textbf{Theorem 1 and of Corollary 1}}
\vskip10pt

For the $W^{2,6}$ bound for $v_j$, we write $v$ for $v_j$ and $\sigma_2$ for
$\sigma_2(e^{-2v_j}g^{-1}_c \circ A_{v_j})=K\circ \varphi_j$, 
and adapt the argument for the $W^{2,p}$ estimates in 
\cite{CGY1} of Chang-Gursky-Yang and push the argument to $p =6$. We will first prove
a $W^{2,3}$ estimate for $v_j$, with the bound depending on an upper  bound of $\sigma_2= K\circ \varphi_j$,
a positive lower bound for $\sigma_2$, and an upper bound for $\int_{\mathbb S^4}|\nabla_0(K\circ \varphi_j)|^2 d vol_{g_c}$.
Then we will extend the $W^{2,3}$ estimate to $W^{2,6}$ estimate for the  $v_j$ in terms of
an upper  bound of $\sigma_2= K\circ \varphi_j$, 
a positive lower bound for $\sigma_2$, and an upper bound for $\int_{\mathbb S^4}|\nabla_0(K\circ \varphi_j)|^4 d vol_{g_c}$.
Since $\int_{\mathbb S^4}|\nabla_0(K\circ \varphi_j)|^4 d vol_{g_c}
= \int_{\mathbb S^4}|\nabla_0 K|^4d vol_{g_c}$, we see that a bound for the
$W^{2,3}$ norm of $v_j$ is given in terms of an upper bound of $K$, a positive lower bound for $K$, and
 an upper bound for $ \int_{\mathbb S^4}|\nabla_0 K|^4$. This will suffice for proving \eqref{ptws}.

\begin{proof}[Proof of the $W^{2,6}$ estimates of \emph{\textbf{Theorem 1}}]
First we list a few key ingredients for these $W^{2,p}$ estimates, mostly adapted from \cite{CGY1}. 
As in \cite{CGY1} we explore two 
differential identities, which in the case of  $\mathbb S^4$, are
\begin{equation} \label{c2}
\begin{split}
S_{ij} \grad^{2}_{ij} R &= 6\, tr E^3 + R|E|^2 + 3 \Delta \sigma_2 + 3(|\grad E|^2 - \frac{|\grad R|^2}{12})\\
&\geq  6\, tr E^3 + \frac{R^3}{12} - 2\sigma_2 R  + 3 \Delta \sigma_2 - \frac{3 |\grad \sigma_2|^2}{2\sigma_2},
\end{split}
\end{equation}
following (5.10) of \cite{CGY1}, with
\[
S_{ij}=\frac{\partial \sigma_2 (A)}{\partial A_{ij}} = - R_{ij} + \frac 12 R g,
\]
and
\begin{equation} \label{V}
\begin{split}
&S_{ij} \grad^{2}_{ij}|\grad v|^2 \\
=&  \frac{R^3}{144} -\frac{tr E^3}{2} - \frac{\sigma_2 R}{12}
 -\frac{R |\grad v|^4}{2} - 2 S_{ij}\grad_i |\grad v|^2 \grad_j v 
 + S_{ij} \grad_l A^{\circ}_{ij} \grad_l v\\
& - 2e^{-2v}S_{ij}\grad_iv \grad_j v + 2Re^{-2v}|\grad v|^2+ \frac{Re^{-4v}}{2}
 -\langle \grad v, \grad \sigma_2 \rangle -2 \sigma_2 e^{-2v},
\end{split}
\end{equation}
following (5.44) of \cite{CGY1} and the fact that $A^0_{ij}=g^0_{ij}$ in the case of $\mathbb S^4$.
Here the differentiations are in the metric $g$.

In \eqref{c2} and \eqref{V} we used $|E|^2 = \frac{R^2}{12} - 2 \sigma_2$ and 
\begin{equation}\label{conca}
|\grad E|^2 - \frac{|\grad R|^2}{12} \geq - \frac{|\grad \sigma_2|^2}{2\sigma_2}. 
\end{equation}
\eqref{conca} can be proven as in (7.26) of \cite{CGY1}, but can also be seen to be based on the general fact that
$\{\sigma_k\}^{1/k}$ is concave in its argument as follows: set $F(A_{ij})= \{\sigma_k (A_{ij})\}^{1/k}$,
then 
\begin{equation}\label{sa}
S_{ij} = \frac{\partial \sigma_k}{\partial A_{ij}} = k F^{k-1} \frac{\partial F}{\partial A_{ij}}, 
\quad \text{and} \quad \grad \sigma_k = S_{ij} \grad A_{ij} =  k F^{k-1} \frac{\partial F}{\partial A_{ij}} \grad A_{ij}.
\end{equation}
So
\[
\grad_l S_{ij} = kF^{k-1}  \frac{\partial^2 F}{\partial A_{ij} \partial A_{IJ}} \grad_l A_{IJ} +
	k(k-1)F^{k-2} \frac{\partial F}{\partial A_{ij}}  \frac{\partial F}{\partial A_{IJ}}  \grad_l A_{IJ}.
\]
Thus
\begin{equation}\label{concap}
\begin{split}
& \sum_l \grad_l S_{ij} \grad_l A_{ij} \\
= & kF^{k-1}  \frac{\partial^2 F}{\partial A_{ij} \partial A_{IJ}} \grad_l A_{IJ}
\grad_l A_{ij} +  k(k-1)F^{k-2} \frac{\partial F}{\partial A_{ij}}  \frac{\partial F}{\partial A_{IJ}}  \grad_l A_{IJ}
\grad_l A_{ij} \\
\le & \frac{(k-1) |\grad \sigma_k|^2}{k \sigma_k} \qquad \qquad \qquad \qquad \text{using concavity of $F$ and \eqref{sa}.}
\end{split}
\end{equation}
In the case of $2k=n=4$, $A_{ij}= E_{ij} + \frac{R}{12} g_{ij}$, and $S_{ij}=\frac{R}{4}g_{ij} -E_{ij}$. So
\[
\sum_l \grad_l S_{ij} \grad_l A_{ij} = \sum_l \{ \frac{\grad_l R}{4}g_{ij} - \grad_l E_{ij}\}\{
\grad_l E_{ij} + \frac{\grad_l R}{12} g_{ij} \} = \frac{|\grad R|^2}{12} -|\grad E|^2,
\]
and by \eqref{concap}
\[
 \frac{|\grad R|^2}{12} -|\grad E|^2 \le  \frac{ |\grad \sigma_2|^2}{2  \sigma_2}.
\]
Because of $S_{ij, j}=0$, which is a consequence of Bianchi identity, 
we can use \eqref{c2} and \eqref{V} to obtain
\[
\begin{split}
0=& \int_{\mathbb S^4} S_{ij} \grad^2_{ij} (R+12|\grad v|^2) \\
&\geq \int_{\mathbb S^4} \frac{R^3}{6} - 6R|\grad v|^4 - 24 S_{ij}\grad_i |\grad v|^2 \grad_j v \\
& + 12 S_{ij}  \grad_l A^{\circ}_{ij} \grad_l v + 24Re^{-2v}|\grad v|^2 -24e^{-2v}S_{ij}\grad_iv \grad_j v\\
& -12\langle \grad v, \grad \sigma_2 \rangle +(6e^{-4v} - 2 \sigma_2) R - 24e^{-2v} \sigma_2 - \frac{3 |\grad \sigma_2|^2}{2\sigma_2} ,
\end{split}
\]   
from which we can estimate $\int_{\mathbb S^4} R^3$ in terms of the other terms:
\begin{equation} \label{r3}
\begin{split}
\int_{\mathbb S^4}\frac{R^3}{6} \le & \int_{\mathbb S^4}  6R|\grad v|^4 + 24 S_{ij}\grad_i |\grad v|^2 \grad_j v \\
& - 12 S_{ij}  \grad_l A^{\circ}_{ij} \grad_l v - 24Re^{-2v}|\grad v|^2 + 24e^{-2v}S_{ij}\grad_iv \grad_j v\\
	& + 12\langle \grad v, \grad \sigma_2 \rangle - (6e^{-4v} - 2 \sigma_2) R  + 24 \sigma_2e^{-2v} +
		\frac{3 |\grad \sigma_2|^2}{2\sigma_2}.
\end{split}
\end{equation}
The integrations are done in the $g$ metric, but due to the $L^{\infty}$ estimates on $v$, the integrals in
$g$ metric are comparable to those in $g_c$. 
 The terms that require careful treatments are
\[
\int_{\mathbb S^4}  S_{ij}\grad_i |\grad v|^2 \grad_j v
\]
and
\begin{equation}\label{rw}
\int_{\mathbb S^4} R|\grad v|^4 \le \left[ \int_{\mathbb S^4} R^3 \right]^{1/3}
\left[ \int_{\mathbb S^4} |\grad v|^6 \right]^{2/3}  \le  \frac{\epsilon}{3} \int_{\mathbb S^4} R^3 
+ \frac{2 \epsilon^{- 1/2}}{3} \int_{\mathbb S^4} |\grad v|^6.
\end{equation}
The term $\int_{\mathbb S^4}  S_{ij}\grad_i |\grad v|^2 \grad_j v$ can be estimated as (5.53) in \cite{CGY1}
\begin{equation} \label{s_ij}
\begin{split}
& \int_{\mathbb S^4}  S_{ij}\grad_i |\grad v|^2 \grad_j v \\
= & - \int_{\mathbb S^4}  |\grad v|^2 S_{ij} \grad^2_{ij} v  \\
=& - \int_{\mathbb S^4}  |\grad v|^2 S_{ij} \{ - \frac{A_{ij}}{2} + \frac{A^0_{ij}}{2} - \grad_i v \grad_j v + \frac{|\grad v|^2}{2}g_{ij} \} \\
=& \int_{\mathbb S^4} |\grad v|^2 \{ \sigma_2 + S_{ij} \grad_i v \grad_j v - \frac{R |\grad v|^2}{2} - 
\frac{S_{ij} A^0_{ij} }{2} \} \\
=& \int_{\mathbb S^4} |\grad v|^2 \{ \sigma_2 - R_{ij} \grad_i v \grad_j v -\frac{S_{ij} A^0_{ij} }{2} \} \\
\le & \int_{\mathbb S^4} |\grad v|^2  \sigma_2 .
\end{split}
\end{equation}
where in the last  line we used  $(R_{ij}) \ge 0$
when $g \in \Gamma_2^+$ in dimension $4$ and $S_{ij} A^0_{ij} \ge 0$ on $\mathbb S^4$.
The terms in the second line of \eqref{r3} can be estimated in terms of $\int_{\mathbb S^4}  Re^{-2v}|\grad v|^2 $,
which in turn can be estimated as
\begin{equation} \label{rw2}
\int_{\mathbb S^4} R e^{-2v} |\grad v|^2  \lesssim \{ \int_{\mathbb S^4}  R^3 \}^{1/3} \{ \int_{\mathbb S^4}  |\grad v|^3 \}^{2/3}
\le \frac{\epsilon}{3} \int_{\mathbb S^4}  R^3  + \frac{2\epsilon^{- 1/2}}{3} \int_{\mathbb S^4} |\grad v|^3.
\end{equation}
The terms in the last line of \eqref{r3} can be estimated in terms of upper bound of
$\sigma_2$, a lower bound of $\sigma_2$, and $\int_{\mathbb S^4} |\nabla \sigma_2|^2$, in a trivial way.
The term $\int_{\mathbb S^4}  |\grad v|^6$ in \eqref{rw} can be estimated as 
\begin{equation}\label{w6}
\int_{\mathbb S^4}  |\grad v|^6 \le \left[ \int_{\mathbb S^4}  |\grad v|^4 \right]^{3/4}
 \left[ \int_{\mathbb S^4}  |\grad v|^{12}\right]^{1/4},
\end{equation}
and, as in (5.73) in \cite{CGY1},
\begin{equation} \label{w12}
\begin{split}
& \left[ \int_{\mathbb S^4}  |\grad v|^{12}\right]^{1/4} \\
\lesssim &\int_{\mathbb S^4} |\grad^2 v|^3 + |\grad v|^6 + e^{-3v} |\grad v|^3 \\
\lesssim &   \int_{\mathbb S^4} R^3 + |\grad v|^6 +1,
\end{split}
\end{equation}
here in the last line we used
\begin{gather}
S_{ij} = S_{ij}^0 + 2 \grad^2_{ij} v - 2 (\Delta v)g_{ij} + 2 \grad_i v \grad_jv + |\grad v|^2 g_{ij},
\\
R = R_0 e^{-2v} - 6 \Delta v + 6 |\grad v|^2, \label{str}\\
\end{gather}
and
\[
0 \le (S_{ij}) \le  (Rg_{ij}).
\]
Using \eqref{w12} in \eqref{w6} and  noting that $\int_{\mathbb S^4}  |\grad v|^{4} $ is small, we obtain
\begin{equation}\label{6R}
 \int_{\mathbb S^4}  |\grad v|^{6} \lesssim \left[ \int_{\mathbb S^4}  |\grad v|^4 \right]^{3/4}
 \int_{\mathbb S^4} R^3  +1,
\end{equation}
Using \eqref{6R}, together with 
\eqref{s_ij}, \eqref{rw2} and \eqref{rw} in \eqref{r3}, and noting the smallness of
$\int_{\mathbb S^4}  |\grad v|^{4} $,  we obtain
an upper bound for $\int R^3$ in terms of $\int |\grad \sigma_2|^2$, upper bound for $\sigma_2$
and positive lower bound for $\sigma_2$. Note that 
$\int_{\mathbb S^4} |\grad \sigma_2|^4=  \int_{\mathbb S^4} |\grad_0(K\circ \phi_j)|^4
=\int_{\mathbb S^4} |\grad_0 K|^4\,\vo$
and using a transformation law like \eqref{6R}, we  can estimate
\[
\int |\Delta_0 v|^3\, \vo \lesssim \int \left(R^3 +|\grad_0 v|^6\right)\, \vo \lesssim \int R^3 +1,
\]
bounded above in terms of $ \int |\grad_0 K|^4 \, \vo$, upper bound for $K$ and positive lower bound for $K$.
Then we can use the $W^{2,p}$ theory for the Laplace operator to obtain the full $W^{2,3}$ estimates for $v$.

\begin{remark}\label{rem:w23}
In fact, for any solution $w$ to \eqref{2}, one can obtain an upper bound for the
$W^{2,3}$ norm of  $w$ in terms of a positive upper and lower bound for $K$, an upper bound for
 $ \int |\grad_0 K|^2 \, \vo$, and an upper bound for $|w|$ and $ \int |\grad_0 w|^4 \, \vo$.
A proof would proceed as above, instead of using the smallness of $ \int |\grad_0 w|^4 \, \vo$
in proving \eqref{6R} and the subsequent bound on $\int R^3$ via \eqref{r3}, one uses
Proposition 5.20, Proposition 5.22, and Lemma 5.24 in
\cite{CGY1} to complete the argument.
\end{remark}

To obtain the $W^{2,6}$ estimates of $v$  by iteration, we multiply \eqref{c2} and \eqref{V} by
$R^p$ and estimate $\int R^{p+3}$ in terms of the other terms:
\[
\int R^p S_{ij} \grad^2_{ij} R \ge 
\int \frac{R^{p+3}}{12} + 6 R^p \, tr E^3 - 2\sigma_2 R^{p+1}  + 3 R^p  \Delta \sigma_2 - \frac{3 R^p  |\grad \sigma_2|^2}{2\sigma_2},
\]
and
\[
\begin{split}
&\int R^p S_{ij} \grad^{2}_{ij}|\grad v|^2 \\
=& \int  \frac{R^{p+3}}{144} -\frac{R^p tr E^3}{2} - \frac{\sigma_2 R^{p+1}}{12}
 -\frac{R^{p+1} |\grad v|^4}{2} -
2 R^p S_{ij}\grad_i |\grad v|^2 \grad_j v \\
& + R^p S_{ij} \grad_l A^{\circ}_{ij} \grad_l v- 2R^p S_{ij}\grad_iv \grad_j v + 2R^{p+1}|\grad v|^2 -
R^p \langle \grad v, \grad \sigma_2 \rangle -2 \sigma_2 R^p + \frac{R^{p+1}}{2}.
\end{split}
\]
From these we obtain
\begin{equation}\label{p+3}
\begin{split}
 & \int \frac{R^{p+3}}{6} \\
\le & \int  R^p S_{ij} \grad^2_{ij} \{ R+ 12 |\grad v|^2 \} - 3 R^p  \Delta \sigma_2 + 6 R^{p+1} |\grad v|^4
+ 24 R^p S_{ij}\grad_i |\grad v|^2 \grad_j v\\
    & + 3 \sigma_2 R^{p+1} + \frac{3 R^p  |\grad \sigma_2|^2}{2\sigma_2}
- 12 R^p S_{ij} \grad_l A^{\circ}_{ij} \grad_l v + 24 R^p S_{ij}\grad_iv \grad_j v - 24 R^{p+1}|\grad v|^2 \\
& + 12 R^p \langle \grad v, \grad \sigma_2 \rangle  + 24 \sigma_2 R^p - 6 R^{p+1}.
\end{split}
\end{equation}

The most crucial terms are
\begin{equation} \label{ps_ijR}
\int R^p S_{ij} \grad^2_{ij} R = -p \int R^{p-1} S_{ij}\grad_i R \grad_j R,
\end{equation}
\begin{equation} \label{ps_ijw}
\begin{split}
&\int R^p S_{ij} \grad^2_{ij} |\grad v|^2 \\
= & - p \int R^{p-1} S_{ij}\grad_i R \grad_j |\grad v|^2 \\
 \le & p \int  R^{p-1} \left[ S_{ij}\grad_i R \grad_j R \right]^{1/2}
 \left[ S_{ij}\grad_i |\grad v|^2 \grad_j |\grad v|^2 \right]^{1/2} \\
\le & p \left[  \int  R^{p-1}  S_{ij}\grad_i R \grad_j R \right]^{1/2}
\left[ \int  R^{p-1}  S_{ij} \grad_i |\grad v|^2 \grad_j |\grad v|^2 \right]^{1/2} \\
\le & \frac{p}{2} \int  R^{p-1}  S_{ij}\grad_i R \grad_j R + 2p
 \int  R^p |\grad^2 v|^2 |\grad v|^2 \\
\le & \frac{p}{2} \int  R^{p-1}  S_{ij}\grad_i R \grad_j R + C \, p 
 \int  R^p (R^2 + |S^0_{ij}|^2 + |\grad v|^2)|\grad v|^2.
\end{split}
\end{equation}
and
\begin{equation}\label{rlap}
\begin{split}
& \int - R^p\Delta \sigma_2  \\
= & p \int R^{p-1} \grad \sigma_2 \grad R \\
\le & p \left[\int |\grad \sigma_2|^4 \right]^{1/4} \left[ \int |\grad R|^2 R^{p-2} \right]^{1/2}
\left[ \int R^{2p} \right]^{1/4} \\
\le &  p \left[\int |\grad \sigma_2|^4 \right]^{1/4} \left[ \int \frac{ R^{p-1}  
S_{ij}\grad_i R \grad_j R}{3 \sigma_2}  \right]^{1/2} \left[ \int R^{2p} \right]^{1/4}\\
\le & \frac{\epsilon p}{2}\int \frac{ R^{p-1}
S_{ij}\grad_i R \grad_j R}{3 \sigma_2} + \frac{p}{2\epsilon} \left[\int |\grad \sigma_2|^4 \right]^{1/2}
\left[ \int R^{2p} \right]^{1/2}
\end{split}
\end{equation}
Next we claim that the Sobolev inequality in dimension $4$ implies 
\begin{equation}\label{sob}
\left[ \int R^{2p}\right]^{1/2} \lesssim p^2 \int  |\grad R|^2 R^{p-2} + \int R^p |\grad v|^2 +\int R^p e^{-2v}.
\end{equation}

Using \eqref{sob} in \eqref{rlap}, we obtain
\begin{equation}\label{rlap2}
\begin{split}
& \int - R^p\Delta \sigma_2  \\
\le & \epsilon p \int \frac{ R^{p-1}
S_{ij}\grad_i R \grad_j R}{3 \sigma_2}  + \frac{\epsilon}{2p} \int  R^p |\grad v|^2 + 
\frac{p^3}{8\epsilon^3} \int |\grad \sigma_2|^4.
\end{split}
\end{equation}

Using \eqref{ps_ijR}, \eqref{ps_ijw}, and \eqref{rlap2} in \eqref{p+3} and choosing $\epsilon >0$
small, we obtain
\begin{equation} \label{p+3rev}
\begin{split}
& \int \frac{R^{p+3}}{6} + \frac{p}{4} R^{p-1}
S_{ij}\grad_i R \grad_j R \\
\lesssim & \int R^{p+1} |\grad v|^4 + R^{p+2}|\grad v|^2 +R^{p}|\grad \sigma_2|^2
 + R^p |\grad v||\grad \sigma_2| + R^{p+1}+ 1  \\
\lesssim & \{ \int R^{p+3} \}^{\frac{p+1}{p+3}}\{\int |\grad v|^{2(p+3)} \}^{\frac{2}{p+3}}
+ \{ \int R^{p+3} \}^{\frac{p+2}{p+3}}\{ \int |\grad v|^{2(p+3)} \}^{\frac{1}{p+3}}  \\
&+ \{ \int R^{2p} \}^{1/2} \{ \int |\grad \sigma_2|^4\}^{1/2}+ \{ \int R^{2p} \}^{1/2} 
\{ \int |\grad v|^4\}^{1/4} \{\int |\grad \sigma_2|^4\}^{1/4} + \int R^{p+1}+ 1
\end{split}
\end{equation}
Now  for $p\le 3$, we have $2p\le p+3$ and $2(p+3)\le 12$. Using the earlier bounds
on $\int R^3$ and $\int |\grad v|^{12}$ from \eqref{w12}, we obtain an upper bound for
$\int R^6$ in terms of $\int |\grad_0 K|^4 \, \vo$,
an upper bound and a positive lower bound of $K$, which again gives a bound for $v$ in $W^{2,6}$.
\end{proof}

\begin{proof}[Proof of  Corollary 1] 
Let $\delta>0$ be small such that the argument for \eqref{6R} and the subsequent $W^{2,3}$ estimate for $v$ via
\eqref{r3} would go through when $\int_{\mathbb S^4}|\nabla v|^4 \le \delta$.
For any admissible solution $w$ to \eqref{2}, Theorem~\ref{chy} implies that there is a constant $B>0$
depending on the $C^2$ norm of $K$, a positive lower  bound of $K$, and $\delta>0$,  such that if $\max w = w(Q) >B$,
then the normalized $v$ defined  as in Theorem~\ref{chy}:
$v=w\circ \varphi + \ln |d\varphi|$ with $ v(Q) =\frac 14 \ln \frac{6}{K(Q)}$, would satisfy
\begin{equation}\label{ene1}
\left| v - \frac 14 \ln \frac{6}{K(Q)}\right| \le \delta, \quad
\text{and}   \quad \int_{\mathbb S^4} |\grad v|^4 \le \delta.
\end{equation}
$v$ also satisfies \eqref{transf} and  then estimate
\eqref{r3}, with $\sigma_2$ standing for $K\circ \varphi$, is valid for $v$.
Then the $W^{2,3}$ estimate in Theorem~\ref{chy} would be valid for $v$, and one obtains
a bound for the $W^{2,3}$ norm of $v$ in terms of an upper bound for $K$, a positive lower bound for $K$,
and an upper bound for $\int_{\mathbb S^4} |\grad_0 K \circ \varphi|^4 \, \vo$, and
since $\int_{\mathbb S^4} |\grad_0 K \circ \varphi|^4 \, \vo= \int_{\mathbb S^4} |\grad_0 K |^4 \, \vo$,
one can use this estimate to obtain the bound for $F[v]$.
Since $I\hskip -0.5mm I[w]= I\hskip -0.5mm I[v]$ and $Y[w]= Y[v]$, the  bound for $F[w]$ now follows.
When the solution $w$ satisfies $w\le B$, then
one can use the Harnack type estimate in \cite{Han04} to obtain a lower bound for $w$, and use
inequality \eqref{inegrad} to obtain an upper bound for $\int_{\mathbb S^4} |\grad w|^4\, \vo$. Then
one can use Remark~\ref{rem:w23} to obtain  the $W^{2,3}$ estimate for $w$ in terms of an upper bound for $K$,
a positive lower bound for $K$, and an upper bound for $\int_{\mathbb S^4} |\grad_0 K|^2 \, \vo$.
Finally  these  $W^{2,3}$ estimates for $w$ give   directly the bound for $F[w]$.
\end{proof}

\end{document}